\documentclass[12pt]{amsart}

\usepackage[margin=1in,marginparwidth=0.8in, marginparsep=0.1in]{geometry}

\pdfoutput=1

\usepackage{etoolbox}

\makeatletter
\let\old@tocline\@tocline
\let\section@tocline\@tocline
\newcommand{\subsection@dotsep}{4.5}
\newcommand{\subsubsection@dotsep}{4.5}
\patchcmd{\@tocline}
  {\hfil}
  {\nobreak
     \leaders\hbox{$\m@th
        \mkern \subsection@dotsep mu\hbox{.}\mkern \subsection@dotsep mu$}\hfill
     \nobreak}{}{}
\let\subsection@tocline\@tocline
\let\@tocline\old@tocline

\patchcmd{\@tocline}
  {\hfil}
  {\nobreak
     \leaders\hbox{$\m@th
        \mkern \subsubsection@dotsep mu\hbox{.}\mkern \subsubsection@dotsep mu$}\hfill
     \nobreak}{}{}
\let\subsubsection@tocline\@tocline
\let\@tocline\old@tocline

\let\old@l@subsection\l@subsection
\let\old@l@subsubsection\l@subsubsection

\def\@tocwriteb#1#2#3{%
  \begingroup
    \@xp\def\csname #2@tocline\endcsname##1##2##3##4##5##6{%
      \ifnum##1>\c@tocdepth
      \else \sbox\z@{##5\let\indentlabel\@tochangmeasure##6}\fi}%
    \csname l@#2\endcsname{#1{\csname#2name\endcsname}{\@secnumber}{}}%
  \endgroup
  \addcontentsline{toc}{#2}%
    {\protect#1{\csname#2name\endcsname}{\@secnumber}{#3}}}%

\newlength{\@tocsectionindent}
\newlength{\@tocsubsectionindent}
\newlength{\@tocsubsubsectionindent}
\newlength{\@tocsectionnumwidth}
\newlength{\@tocsubsectionnumwidth}
\newlength{\@tocsubsubsectionnumwidth}
\newcommand{\settocsectionnumwidth}[1]{\setlength{\@tocsectionnumwidth}{#1}}
\newcommand{\settocsubsectionnumwidth}[1]{\setlength{\@tocsubsectionnumwidth}{#1}}
\newcommand{\settocsubsubsectionnumwidth}[1]{\setlength{\@tocsubsubsectionnumwidth}{#1}}
\newcommand{\settocsectionindent}[1]{\setlength{\@tocsectionindent}{#1}}
\newcommand{\settocsubsectionindent}[1]{\setlength{\@tocsubsectionindent}{#1}}
\newcommand{\settocsubsubsectionindent}[1]{\setlength{\@tocsubsubsectionindent}{#1}}

\renewcommand{\l@section}{\section@tocline{1}{\@tocsectionvskip}{\@tocsectionindent}{}{\@tocsectionformat}}%
\renewcommand{\l@subsection}{\subsection@tocline{2}{\@tocsubsectionvskip}{\@tocsubsectionindent}{}{\@tocsubsectionformat}}%
\renewcommand{\l@subsubsection}{\subsubsection@tocline{3}{\@tocsubsubsectionvskip}{\@tocsubsubsectionindent}{}{\@tocsubsubsectionformat}}%
\newcommand{\@tocsectionformat}{}
\newcommand{\@tocsubsectionformat}{}
\newcommand{\@tocsubsubsectionformat}{}
\expandafter\def\csname toc@1format\endcsname{\@tocsectionformat}
\expandafter\def\csname toc@2format\endcsname{\@tocsubsectionformat}
\expandafter\def\csname toc@3format\endcsname{\@tocsubsubsectionformat}
\newcommand{\settocsectionformat}[1]{\renewcommand{\@tocsectionformat}{#1}}
\newcommand{\settocsubsectionformat}[1]{\renewcommand{\@tocsubsectionformat}{#1}}
\newcommand{\settocsubsubsectionformat}[1]{\renewcommand{\@tocsubsubsectionformat}{#1}}
\newlength{\@tocsectionvskip}
\newcommand{\settocsectionvskip}[1]{\setlength{\@tocsectionvskip}{#1}}
\newlength{\@tocsubsectionvskip}
\newcommand{\settocsubsectionvskip}[1]{\setlength{\@tocsubsectionvskip}{#1}}
\newlength{\@tocsubsubsectionvskip}
\newcommand{\settocsubsubsectionvskip}[1]{\setlength{\@tocsubsubsectionvskip}{#1}}

\patchcmd{\tocsection}{\indentlabel}{\makebox[\@tocsectionnumwidth][l]}{}{}
\patchcmd{\tocsubsection}{\indentlabel}{\makebox[\@tocsubsectionnumwidth][l]}{}{}
\patchcmd{\tocsubsubsection}{\indentlabel}{\makebox[\@tocsubsubsectionnumwidth][l]}{}{}

\newcommand{\@sectypepnumformat}{}
\renewcommand{\contentsline}[1]{%
  \expandafter\let\expandafter\@sectypepnumformat\csname @toc#1pnumformat\endcsname%
  \csname l@#1\endcsname}
\newcommand{\@tocsectionpnumformat}{}
\newcommand{\@tocsubsectionpnumformat}{}
\newcommand{\@tocsubsubsectionpnumformat}{}
\newcommand{\setsectionpnumformat}[1]{\renewcommand{\@tocsectionpnumformat}{#1}}
\newcommand{\setsubsectionpnumformat}[1]{\renewcommand{\@tocsubsectionpnumformat}{#1}}
\newcommand{\setsubsubsectionpnumformat}[1]{\renewcommand{\@tocsubsubsectionpnumformat}{#1}}
\renewcommand{\@tocpagenum}[1]{%
  \hfill {\mdseries\@sectypepnumformat #1}}

\let\oldappendix\appendix
\renewcommand{\appendix}{%
  \leavevmode\oldappendix%
  \addtocontents{toc}{%
    \protect\settowidth{\protect\@tocsectionnumwidth}{\protect\@tocsectionformat\sectionname\space}%
    \protect\addtolength{\protect\@tocsectionnumwidth}{2em}}%
}
\makeatother

\makeatletter
\settocsectionnumwidth{2em}
\settocsubsectionnumwidth{2.5em}
\settocsubsubsectionnumwidth{3em}
\settocsectionindent{1pc}%
\settocsubsectionindent{\dimexpr\@tocsectionindent+\@tocsectionnumwidth}%
\settocsubsubsectionindent{\dimexpr\@tocsubsectionindent+\@tocsubsectionnumwidth}%
\makeatother

\settocsectionvskip{5pt}
\settocsubsectionvskip{0pt}
\settocsubsubsectionvskip{0pt}
    
\settocsectionformat{\bfseries}
\settocsubsectionformat{\mdseries}
\settocsubsubsectionformat{\mdseries}
\setsectionpnumformat{\bfseries}
\setsubsectionpnumformat{\mdseries}
\setsubsubsectionpnumformat{\mdseries}

\let\oldtableofcontents\tableofcontents
\renewcommand{\tableofcontents}{%
  \vspace*{-\linespacing}
  \oldtableofcontents}

\usepackage{soul}

\usepackage{amsfonts,amssymb,latexsym,amsmath,graphicx}
\input xy
\xyoption{all}

\usepackage[backref=page, bookmarks=true, bookmarksopen=true,%
bookmarksdepth=3,bookmarksopenlevel=2,%
colorlinks=true,%
linkcolor=blue,%
citecolor=blue,%
filecolor=blue,%
menucolor=blue,%
urlcolor=blue]{hyperref}

\usepackage{cleveref}

\usepackage[all, knot, poly]{xy}
\xyoption{knot}

\usepackage{times}
\usepackage{mathrsfs}

\newtheorem{lemma}{Lemma}[section]
\newtheorem{proposition}[lemma]{Proposition}
\newtheorem{corollary}[lemma]{Corollary}
\newtheorem{theorem}[lemma]{Theorem}

\theoremstyle{definition}
\newtheorem{definition}[lemma]{Definition}

\theoremstyle{remark} 
\newtheorem{remark}[lemma]{Remark}

\newtheorem{example}[lemma]{Example}

\newcommand{\In}{\subset}
\newcommand{\RM}{\backslash}
\newcommand{\into}{\hookrightarrow}
\renewcommand{\d}{\partial}
\def\congto{\xrightarrow{\sim}}

\usepackage{tikz,tikz-cd}

\newcommand{\A}{\mathbb{A}}

\newcommand{\C}{\mathbb{C}}

\newcommand{\G}{\mathbb{G}}

\newcommand{\R}{\mathbb{R}}

\newcommand{\T}{\mathbb{T}}

\newcommand{\Z}{\mathbb{Z}}

\newcommand{\cB}{\mathcal{B}}

\newcommand{\cH}{\mathcal{H}}
\newcommand{\cI}{\mathcal{I}}
\newcommand{\cJ}{\mathcal{J}}

\newcommand{\cO}{\mathcal{O}}

\newcommand{\cX}{\mathcal{X}}
\newcommand{\cY}{\mathcal{Y}}

\newcommand{\core}{\mathfrak{C}}

\newcommand{\fg}{\frak{g}}

\newcommand{\TAo}{\mht_{\operatorname{basic}}}
\newcommand{\TAno}{\TAo^n}

\newcommand{\sz}{z}
\newcommand{\sw}{w}

\newcommand{\mht}{\frak{U}}
\newcommand{\mhtcover}{\widetilde{\mht}}
\newcommand{\chamberset}{\Xi}
\newcommand{\tchamberset}{\widetilde{\chamberset}}
\newcommand{\Bx}{\mathbf{x}}
\newcommand{\nhd}{U}

\newcommand{\GIT}{\theta}

\newcommand{\hypermoment}{\mu^{\cflavor}}

\newcommand{\tBx}{\widetilde{\Bx}}

\newcommand{\sslash}{/\!/\!}
\newcommand{\reg}{\operatorname{reg}}
\newcommand{\OmegaBet}{\Omega}
\newcommand{\rk}{\operatorname{rk}}

\newcommand{\parone}{\beta}

\newcommand{\flavor}{\G}
\newcommand{\flavorparone}{\flavor_\parone}
\newcommand{\cflavor}{G}
\newcommand{\gauge}{\T}
\newcommand{\cgauge}{T}
\newcommand{\cflavorparone}{G_\parone}
\newcommand{\liegauge}{\frak{g}}

\newcommand{\Dtor}{(\C^*)^n}
\newcommand{\Dtordual}{(\C^*)^n}
\newcommand{\Dtordualcpt}{U(1)^n}

\DeclareMathOperator{\Loc}{Loc}
\DeclareMathOperator{\Fuk}{Fuk}
\DeclareMathOperator{\Coh}{Coh}

\DeclareMathOperator{\End}{End}
\DeclareMathOperator{\Spec}{Spec}

\author{Michael McBreen, Vivek Shende, and Peng Zhou}

\title{The Hamiltonian reduction of hypertoric mirror symmetry}

\begin{document}

\begin{abstract}
We give a  `Fukaya category commutes with reduction' theorem for the Hamiltonian torus action on a multiplicative hypertoric variety. 
\end{abstract}

\maketitle

\section{Introduction}

Let us recall two basic but fundamental examples of homological mirror symmetry.

\begin{example} \label{intro pants example}
The pair of pants $\C^* \setminus 1$ is mirror to the coordinate axes $\{xy = 0\}$ \cite{seidel-genustwo, AAEKO, nadler-pants}. 
While extremely simple, this example provides the model case for many investigations in higher dimensional  settings
\cite{seidel-quartic, sheridan-pants, nadler-pants, gammage-shende-hypersurfaces, gammage-shende-fanifold}.  
\end{example}

\begin{example} \label{intro basic MCV example}
$\C^2 \setminus  \{1+xy = 0\}$ is mirror to itself. 
This space contains two exact Lagrangian tori 
(`Clifford' and `Chekanov') whose Floer-theoretic relationship is a first illustration of 
wall crossing or cluster transformation \cite{auroux-anticanonical, seidel-dynamics}.  It correspondingly provides the building block or model calculation for the various appearances of
cluster structures in Fukaya categories, e.g. \cite{gross-hacking-keel-kontsevich, shende-treumann-williams-zaslow}.  
This space is also the basic instance of a multiplicative hypertoric variety, and provides the model for their homological mirror symmetry as established in \cite{gammage-mcbreen-webster}.      
\end{example}

Our purpose in the present article is to explain and generalize a relation between the two examples above.  

\vspace{2mm}

Consider the map
\begin{eqnarray} \label{basic space projection}
\pi: \C^2 \setminus  \{1+xy = 0\} & \to & \C^* \\
(x, y) & \mapsto & 1 + xy \nonumber
\end{eqnarray} 
Note that the regular locus of $\pi$ is $\C^* \setminus 1$, and $\pi^{-1}(1) = \{xy = 0\}$.  
The map $\pi$ can be understood as the quotient by the $\C^*$-action which scales $x$ and $y$ with opposite weights, or equivalently, as the symplectic reduction with respect to the Hamiltonian $U(1)$ action with moment map $|x|^2 - |y|^2$.

\vspace{2mm} 
As suggested by Teleman \cite{teleman-icm}, in the presence of such an action, one may hope for 
a `Fukaya category commutes with reduction' result.

First taking Fukaya category, then taking reduction, 
means the following.  A Hamiltonian action of a Lie group $G$ on a symplectic manifold $\cX$ is expected to correspond to `topological $G$-action on $\Fuk(\cX)$', which gives in particular an enrichment of $\Fuk(\cX)$ over the chains over the loop space of the group, $C_* \Omega G$. For wrapped Fukaya categories of Liouville manifolds, which is the case of interest to us, such a structure was constructed in \cite{oh-tanaka}. 

Let $C_* \Omega G \to \Z$ be the augmentation given by the (monodromies of the) constant rank-one local system $\Z_G$.
`Reduction' means forcing 
$C_* \Omega G$ to act in the same way that it acts on the trivial local system, i.e., forming: 
\begin{equation} \Fuk(\cX)  \otimes_{C_* \Omega G} \Z,\end{equation} 
Note that $C_* \Omega U(1) = \Z[t, t^{-1}]$, and the augmentation map is $t \to 1$.  

 Across a mirror symmetry
equivalence 
$$\Fuk(\cX) = \Coh(\cY),$$
we transport such enrichment to 
$$\Z[t, t^{-1}] \to HH^0(\Fuk(\cX)) = HH^0(\Coh(\cY)) = \Gamma(\cY, \cO_{\cY}).$$
That is, we have a map $\Spec( \Gamma(\cY, \cO_{\cY}) ) \to \G_m$.  Composing
with the affinization morphism $\cY \to \Spec(\Gamma(\cY, \cO_{\cY}))$, we recover a map $\pi_{G \circlearrowright \cX}: \cY \to \G_m$.  
In particular, 
$$\Fuk(\cX) \otimes_{\Z[t, t^{-1}]} \Z = 
\Coh(\cY) \otimes_{\Z[t, t^{-1}]} \Z = 
\Coh(\pi_{G \circlearrowright \cX}^{-1}(1)).$$

Note conversely that from any map $\pi_\cY: \cY \to \G_m$, one obtains an enrichment of $\Coh(\cY)$ over $k[t, t^{-1}]$ by pullback of functions, and, by transport of structure across a mirror symmetry, a corresponding enrichment of $\Fuk(\cX)$.

\vspace{2mm}
Let us consider the mirror symmetry 
$\Fuk(\cX) = \Coh(\cY)$ for
$\cX = \{1+xy = 0\} = \cY$. 
Accept for the moment that $\pi_{U(1) \circlearrowright \cX}: \cY \to \G_m$ is precisely the map $\pi$ given in  
\eqref{basic space projection} above.  Then, the reduction of the Fukaya category is computed by: 
\begin{equation} \Fuk(\C^2 \setminus \{1+xy = 0\}) \otimes_{\Z[t, t^{-1}]} \Z =  \Coh(\pi^{-1}(1)) = \Coh(\{xy = 0\}). 
\end{equation}
That is to say, the reduction of the Fukaya category appearing in Example \ref{intro basic MCV example} gives one side of the mirror symmetry of Example \ref{intro pants example}. 

\vspace{2mm} 

On the other hand, let us contemplate the Fukaya category of the reduction.  A problem: zero is not a regular point for the moment map for the $U(1)$ action, so the reduction does not naturally carry a symplectic form -- and besides, the reduced space is topologically a cylinder, which does not have the correct Fukaya category.  
Instead, we use the following ad-hoc prescription.
\begin{definition}
    Let $\cX /\!/_{reg}\, G$ denote the quotient by $G$ of the smooth part of the moment fiber $\mu_G^{-1}(0)^{sm}/G$.
\end{definition} 
Then
$$\{\C^2 \setminus \{1 + xy = 0 \}\} /\!/_{reg} \, U(1) = \C^* \setminus 1$$
Observe: $\Fuk(\C^* \setminus 1)$ is the other side of the mirror symmetry of Example \ref{intro pants example}. 

\vspace{2mm}
While we do not know how to motivate the operation $/\!/_{reg}$ from first principles (and even less so the prescription of wrapping around the deleted locus), we will show that this prescription works in the more general setting of multiplicative hypertoric varieties.  

Recall that a $2n$ complex dimensional multiplicative hypertoric $\cX$ naturally comes with a Hamiltonian torus action by some $T_\cX$, the symplectic reduction by which coincides with an algebraic map $\pi_{\cX}:  \cX \to T_{\cX, \C}^\vee$.

\begin{theorem} \label{multiplicative main theorem}
Let $\cX$ be a multiplicative hypertoric variety, and $\cY$ its \cite{gammage-mcbreen-webster} mirror, so that
\begin{equation} \label{htms}
\Fuk(\cX) \cong \Coh(\cY)
\end{equation} 
Then there is a Liouville structure on 
$\cX /\!/_{reg} \, T_{\cX}$ such that 
\begin{equation} \label{downstairs}
\Fuk(\cX /\!/_{reg} \, T_{\cX}) = \Coh(\pi_{\cY}^{-1}(1)).
\end{equation}
\end{theorem} 

We prove Theorem \ref{multiplicative main theorem} by passing via \cite{GPS3} to microsheaves on skeleta for both $\cX$ and $\cX /\!/_{reg} \, T_{\cX}$. 
Ultimately we will show that 
 both skeleta map to a hyperplane arrangement, and that the pushforwards of the respective categories of microsheaves can be  related as expected by the
$\Z[T_{\cY, \C}^\vee]$-linear structure 
translated through \eqref{htms}.
We first give an  algebraic description of the  resulting structures and their relations in Section \ref{two sheaves}.
The calculation of skeleta and microsheaves
for $\cX$ was done in \cite{gammage-mcbreen-webster}; we recall the results in Section \ref{section mht}. 
In Section \ref{section downstairs skeleton} we calculate the skeleton and microsheaves for $\cX /\!/_{reg} \, T_{\cX}$ along similar lines as \cite{gammage-shende-hypersurfaces, Zhou-skeleton, gammage-shende-fanifold}. 
Finally we collect the results and prove the theorem in Section \ref{section proof of main theorem}. 

\vspace{2mm}
To understand \ref{multiplicative main theorem} as a `Fukaya category commutes with reduction' result, we would like to know that the $\Z[T_{\cY, \C}^\vee]$-linear structure imported via mirror symmetry agrees with some intrinsically defined $C_* \Omega T_{\cX} = \Z[T_{\cX, \C}^\vee]$-linear structure on the Fukaya category.  (This makes sense because under the mirror symmetry from \cite{gammage-mcbreen-webster}, there is a natural indentification of $T_{\cY, \C}^\vee = T_{\cX, \C}^\vee$.) 

Such a structure is constructed in \cite{oh-tanaka} for Hamiltonian group actions on Liouville manifolds. We will sketch in Section \ref{section right action} a similar construction of the action which interacts more straightforwardly with the sheaf-Fukaya comparison of \cite{GPS3}.  In terms of this action, we prove:

\begin{theorem} \label{theorem right action}
    The mirror symmetry of \cite{gammage-mcbreen-webster} intertwines the  $\Z[T_{\cX, \C}^\vee]$ enrichment of $\Fuk(\cX)$ with the $\Z[T_{\cY, \C}^\vee]$ enrichment of $\Coh(\cY)$. 
\end{theorem}

\vspace{2mm}
{\bf Acknowledgements.}
This work is a spin-off of ongoing joint work with Mina Aganagic around mirror symmetry for Coulomb branches of quiver gauge theories (of which the hypertoric varieties are the `abelian' cases), developing from her works \cite{aganagic-knotcat-I, aganagic-knotcat-II}. We thank her for many discussions and much input.  In particular, the main results of the article stem from conjectures made in  collaboration with Aganagic circa 2019.  Let us also note that during the (lengthy) genesis of this article, other authors have independently arrived at related ideas \cite{Lekili-Segal, LLLM}.

V.S. is supported by Novo Nordisk Foundation grant NNF20OC0066298, Villum Fonden Villum Investigator grant 37814, and Danish National Research Foundation grant DNRF157.

\section{Two sheaves of categories on the line} \label{two sheaves}

\begin{definition}
There is a constructible cosheaf of algebras $\cB_0$ on the stratified space $(\R, 0)$ with stalk at the origin 
given by  
$$Paths( \bullet \overset{x}{\underset{y}\rightleftarrows} \bullet) / (xy = 0 = yx)$$
and the corestrictions to right and left
are given by inclusion of the nodes. 
\end{definition}

\begin{definition}
There is a constructible cosheaf of algebras $\cB$ on the stratified space $(\R, 0)$ with stalk at the origin 
given by:
$$Paths( t \, \rotatebox[origin=c]{270}{$\circlearrowleft$} \,\, \bullet \overset{x}{\underset{y}\rightleftarrows} \bullet \,\, \rotatebox[origin=c]{90}{$\circlearrowright$} \, \tau) 
[t^{-1}, \tau^{-1}] / (t = 1+yx, \tau = 1 + xy) $$
and the corestrictions to left and right given by the inclusion of subalgebras $\Z[t, t^{-1}]$ and $\Z[\tau, \tau^{-1}]$. 

This quiver was originally introduced by Beilinson in order to organize properties of nearby and vanishing cycles \cite{Beilinson-glue}, or correspondingly to describe the category of sheaves on the $\C$, constructible with respect to the stratification $\C = \C^* \sqcup 0$. 
\end{definition} 

\begin{proposition} \label{center of beilinson quiver}
There is an isomorphism from $\Z[s, s^{-1}]$ to the center of $\cB$, restricting  on one side of the origin to the isomorphism
$\Z[s, s^{-1}] = \Z[t, t^{-1}]$ via $s \to t$, and on the other side correspondingly with $s \to \tau$. 
\end{proposition}
\begin{proof}
An element of the center is a natural transformation from the identity to itself; i.e. an endomorphism of each object, commuting with all maps of objects.  Given
any element $f(s) \in \Z[s, s^{-1}]$, one can obtain such an automorphism is given by multiplying by $f(t)$ on the left node and $f(\tau)$ on the right.  
The relations imply that this commutes with all possible paths.  

Let us see this is an isomorphism.  Let $z$ be a central element, and let $e_t$ be the idempotent at the first vertex. We have $ze_t = f(t)e_t$ for some $f(t) \in \Z[t,t^{-1}]$. Similarly $z e_\tau = g(\tau) e_\tau$ for some $g(\tau) \in \Z[\tau,\tau^{-1}]$. Now consider the product $f(t)y=zy=yz=yg(\tau)$. Thus $xf(1+yx)y = xyg(1 + xy)$. Expanding in powers of $xy$, we find $f = g$.
\end{proof}

\begin{proposition} \label{reduction of beilinson quiver} 
Consider the map $\Z[s, s^{-1}] \to \Z$ sending $s \mapsto 1$.  Using this map, one has
$$\cB \otimes_{\Z[s, s^{-1}]} \Z   = \cB_0$$
\end{proposition} 
\begin{proof}
Obvious.
\end{proof}

Consider now a smooth manifold $M$ and simple normal crossing co-oriented arrangement of closed codimension one submanifolds
$\cH \subset M$.  Then there is naturally associated constructible cosheaf of algebras $\cB_0(M, \cH)$ which, locally at the intersection of $n$ hyperplanes, 
is a pullback of $\cB_0^{\boxtimes n}$.

To correspondingly globalize $\cB$, it is natural to ask that $M$ is an integral affine manifold; in particular carrying a local system $L$ of lattices, 
and isomorphism $T^*M \cong L \otimes \R$. 
In addition, we demand that at any stratum, the co-orienting co-vectors to the hyperplanes passing through the stratum form part of an integer basis for $L$.   This data allows us in a natural way to define 
$\cB(M, \cH)$, which, locally at an intersection of $n$ hyperplanes, is isomorphic to $\cB^{\boxtimes n} \boxtimes \Z[s, s^{-1}]^{\boxtimes \dim M - n}$, 
and which admits a central map 
\begin{equation} \label{central action on B}
\Z[L] \to \cB(M, \cH)
\end{equation}

Globalizing Proposition \ref{reduction of beilinson quiver} gives

\begin{corollary} \label{global algebra cosheaf reduction} 
$\cB(M, \cH) \otimes_{\Z[L]} \Z = \cB_0(M, \cH)$.
\end{corollary}

\section{homological mirror symmetry for multiplicative hypertoric varieties} \label{section mht}

In this section we recall some basic facts about multiplicative hypertoric spaces and their Liouville geometry, along with the mirror symmetry results of \cite{gammage-mcbreen-webster}.

\subsection{Combinatorial data}
Given a complex torus $\A$, we write $A$ for the maximal compact subtorus. 

Consider a short exact sequence of complex tori
\begin{equation} \label{eq:basictoriseq}
	1\longrightarrow \gauge \longrightarrow (\C^*)^n \longrightarrow \flavor \longrightarrow 1
\end{equation}
along with the dual sequence 
\begin{equation} \label{eq:dualbasictoriseq} 1\longrightarrow \flavor^{\vee}\longrightarrow (\C^*)^n \longrightarrow \gauge^{\vee}\longrightarrow 1.\end{equation}
The fibers of the map $(\C^*)^n \to \gauge^{\vee}$ are a family of torsors over $\flavor^{\vee}$.  We assume throughout that no coordinate subtorus lies in the image of $\gauge$.

\begin{definition}
Given $\parone \in \gauge^{\vee}$, let $\flavorparone^{\vee} \subset (\C^*)^n$ be the preimage of $\parone$. We define a toric hyperplane arrangement $\{H_i\} \subset \flavorparone^{\vee}$, given by the intersections with the coordinate subtori of  $\Dtordual$. Define $\flavor^{\vee, \circ}_\parone \subset \flavorparone^{\vee}$ be the complement of this arrangement.
\end{definition} 

Now suppose moreover that $\beta$ lies in the compact subtorus $\cgauge^{\vee}$. The intersection $\cflavorparone^{\vee} := \flavorparone^{\vee} \cap \Dtordualcpt$ is naturally an integral affine manifold with integral structure induced from $\Dtordualcpt = \R^n/\Z^n$. Intersecting with the $H_i$ defines an integral normal crossings arrangement \begin{equation} \label{hypertoric hyperplane arrangement}
    \cH \subset \cflavorparone^{\vee}.
\end{equation} 

The combinatorial type of this arrangement is locally constant as $\beta$ varies in $\cgauge^{\vee}$, away from a finite set of codimension-one subtori.

\subsection{Construction}

Consider  
$$\TAo := \{ (\sz, \sw ) | \sz \sw \neq -1 \} \subset \C^2$$ 
with a holomorphic symplectic form 
$$\OmegaBet := \frac{d \sz d \sw}{1 + \sz \sw}.$$

The map 
\begin{eqnarray*}
	\mu_{\C^*} :  \TAo & \to & \C^* \\
	(\sz, \sw) & \mapsto & 1 + \sz \sw
\end{eqnarray*}
is the multiplicative moment map for the (complex) quasi-Hamiltonian $\C^*$-action scaling $\sz$ (resp. $\sw$) with weight $1$ (resp. $-1$). Concretely, this means that $d\log(\mu_{\C^*})$ is the contraction of $\OmegaBet$ by the vector field associated to the $\C^*$ action.

The torus $\Dtor$ acts factorwise on the symplectic variety $\TAno$, with moment map $$\mu_{\C^*}^n :  \TAno \to (\C^*)^n =  \Dtordual.$$ By restricting along the inclusion $\gauge \hookrightarrow \Dtor,$ we also have an action of $\gauge$ and  a moment map 
$$ \mu^\gauge : \TAno \to \gauge^\vee. $$ This moment map is obtained by composing the moment map for $\Dtor$ with the the map $ \Dtordual \to \gauge^{\vee}$ induced by the inclusion $\gauge \to \Dtor.$ Now let $\parone\in \gauge^\vee, \GIT \in \fg^{\vee}_\Z.$

\begin{definition}\label{def:mht}
	The Betti hypertoric variety $\mht_{\parone,\GIT}$ associated to the data of the short exact sequence \eqref{eq:basictoriseq} and $(\parone,\GIT)$ 
	is the GIT quotient $(\mu^{\gauge})^{-1}(\parone)/\!\!/_{\GIT} \gauge$. 
\end{definition}

For $\parone$ away from a finite number of complex subtori of $\gauge^{\vee}$, $\mht_{\parone,0}$ is a smooth affine variety of complex dimension $2 \rk \flavor$. Fix $\parone$ generic in this sense, but lying in $\cgauge^{\vee} \subset \gauge^{\vee}$. 

The action of $(\C^*)^n$ on $\TAno$ descends to an action on $\mht_{\parone, \GIT}$ factoring through the torus $\flavor = \Dtor/\gauge$. The latter has a multiplicative moment map 
\begin{equation} \label{multiplicative moment map}
\mu^\flavor : \mht_{\parone, \GIT} \longrightarrow \Dtordual 
\end{equation}
which is just the restriction of the corresponding map for $\Dtor$. 

\begin{proposition} \label{prop:complexflavormoment}
	$\mu^\flavor : \mht_{\parone, 0} \to \Dtordual$ has image $\flavorparone^{\vee}$ and regular locus $\flavor^{\vee, \circ}_\parone$.
\end{proposition}

\subsection{Liouville structure}
One can define an open retract $\mht^{<1}_{\parone, 0} \subset \mht_{\parone, 0}$ carrying a (non-complete) hyperk\"ahler metric $g$. We denote the triplet of complex structures by $I,J,K$, the k\"ahler forms by $\omega_I, \omega_J, \omega_K$ and set $\Omega_I := \omega_J + I \omega_K$.
The symplectic manifold $(\mht^{<1}_{\parone, 0}, \omega_J)$ is made into a Liouville domain in \cite{gammage-mcbreen-webster}, with Liouville completion given by $\mht_{\parone, 0}$.\footnote{
The Liouville structure constructed in \cite{gammage-mcbreen-webster} is expected, but has not been shown, to be deformation equivalent to the one arising from the Stein structure.} 

We write $\mhtcover_{\parone,0}$ and $\mhtcover^{<1}_{\parone,0}$ for the universal covering spaces. The action of the compact subtorus $\cflavor \subset \flavor$ on the universal cover is hyperhamiltonian with moment map $\widetilde{\hypermoment} : \mhtcover^{<1}_{\parone,0} \to \liegauge^{\vee}_\C \oplus \liegauge^{\vee}_{\R, \parone}$. Here $\liegauge^{\vee}_{\R, \parone}$ is the preimage of $\beta$ under the projection $\R^n \to \frak{g}^{\vee}_\R$. 

The lattice $\liegauge^{\vee}_\Z \cong \pi_1(\mht_{\parone, 0})$ acts by deck transformations. The moment map is equivariant for the action of $\liegauge^{\vee}_\Z$ on $\liegauge^{\vee}_\C \oplus \liegauge^{\vee}_\R$ by translation of the second summand. Taking the quotient, we get a map
$\hypermoment : \mht^{<1}_{\parone, 0} \to \liegauge^{\vee}_\C \times \cflavorparone^{\vee}$. 

The moment map for the action of $\cflavor$ on $(\mht_{\parone, 0}, \omega_J)$ extends the composition $\hypermoment_\R : \mht_{\parone, 0}^{<1} \to \liegauge^{\vee}_\C \times \cflavorparone^{\vee} \to \liegauge_\R^{\vee}$ obtained by projection onto the real part of the first factor, and we also denote it $\hypermoment_\R$. Recall that $\mht_{\parone, 0} \sslash_{reg} \cflavor$ is defined as the quotient of the smooth part of $(\hypermoment_\R)^{-1}(0)$ by $\cflavor$.

 \begin{proposition}
     $\mu^\flavor : \mht_{\parone, 0} \to \Dtordual$ induces an isomorphism of $\mht_{\parone, 0} \sslash_{\reg} \cflavor $ with $\flavor^{\vee, \circ}_\parone$.
\end{proposition} 
\begin{proof}
The fibers of $\mu^{\flavor} : \mht_{\parone, 0} \to \flavorparone^{\vee}$ are products of $\C^*$ and $\C \cup_0 \C$. The torus $\flavor$ acts on any such fiber with a finite number of orbits; the action of the subtorus $\cflavor$ is thus free away from the singular locus. Thus the singularities of the $\cflavor$-moment fibers are located along the singular locus of the complex moment fibers. The claim then follows from Proposition \ref{prop:complexflavormoment}.
\end{proof}
We write $\core \subset \mht_{\parone, 0}$ for the Liouville skeleton, and $\widetilde{\core}$ for its preimage in $\mhtcover_{\parone,0}$. In \cite{gammage-mcbreen-webster}, it is shown that $\core = (\hypermoment)^{-1}(0 \times \cflavorparone^{\vee})$. This allows a very explicit description of the skeleton, which we now recall.

Let $\chamberset$ index the chambers of the arrangement $\cH$ on $\cflavorparone^{\vee}$, and $\tchamberset$ index the chambers of the periodic arrangement $\widetilde{\cH}$ on $\liegauge_\R^{\vee}$. The group $\liegauge^{\vee}_\Z = \pi_1(\cflavor^{\vee})$ acts on $\tchamberset$ by translations with quotient $\chamberset$.

\begin{proposition} \label{prop:coregeom}
$\widetilde{\core}$ is a union of smooth $\cflavor$-stable Lagrangians $\frak{X}_{\tBx}$, indexed by $\tBx \in \tchamberset$. Each $\frak{X}_{\tBx}$ is a complex toric variety in the complex structure $I$, (complex) Lagrangian with respect to $\Omega_I$. The action of the deck transformation group $\pi_1(\mht_{\parone, 0}) \cong \pi_1(\cflavor^{\vee})$ on the component set is described by the natural action of $\liegauge^{\vee}_\Z$ on $\tchamberset$.
\end{proposition}

We denote the image of $\frak{X}_{\tBx}$ in $\mht_{\parone, 0}$ by $\frak{X}_{\Bx}$. In general, this is an immersed toric submanifold.

\begin{proposition}
Each of the lagrangians $\frak{X}_{\tBx}$ is contained in a Weinstein neighborhood $\nhd_{\tBx}$, such that $\widetilde{\core} \cap \nhd_{\tBx}$ is identified the union of conormals to the toric strata of $\frak{X}_{\tBx}$.
\end{proposition}

In \cite{gammage-mcbreen-webster}, Maslov data is given for $\mht_{\parone, 0}$ pulling back to the canonical cotangent Maslov data on each $\nhd_{\tBx}$. For this choice, we have the following. 

\begin{proposition} \label{prop:ushonweinsteinnhd} \cite{gammage-mcbreen-webster}
$\mu sh_{\core}$ is equivalent, after pullback to $\nhd_{\tBx}$, to constructible sheaves on the toric variety $\frak{X}_{\tBx}$ smooth along the toric stratification. 
\end{proposition}

We now consider the compact-torus valued factor $\mu^{\cflavor}_{U(1)} : \core \to \cflavorparone^{\vee}$ of the moment map, restricted to the core. The pushforward $(\mu^{\cflavor}_{U(1)})_* \mu sh_{\core}$ is a cosheaf of categories locally constant with respect to the natural stratification induced by the integral crossings arrangement $\cH$.

Consider a chamber $\Delta_{\Bx}$ of this arrangement, corresponding to a smooth irreducible component $\frak{X}_{\Bx} \subset \core$. The preimage under $\mu^{\cflavor}_{U(1)}$ of a small neighborhood $V_{\Bx} \supset \Delta_{\Bx}$ coincides with $\nhd_{\tBx} \cap \core$, possibly after shrinking the latter. 

The abelian category of perverse sheaves on $\frak{X}_{\Bx}$ smooth along the toric stratification is equivalent to representations of a certain quiver with relations, as can be deduced readily from \cite{Beilinson-glue}, and is described explicitly in Dupont's thesis \cite{Dupont-thesis}. Combining with \ref{prop:ushonweinsteinnhd}, we obtain a quiver description of $(\mu^{\cflavor}_{U(1)})_* \mu sh_{\core}$ restricted to $V_{\Bx}$. This globalizes as follows:
\begin{theorem} \label{GMW mush to quiver} \cite{gammage-mcbreen-webster}. 
There is an equivalence of sheaves of categories 
\begin{equation} (\mu^{\cflavor}_{U(1)})_* \mu sh_{\core}  \cong \cB(\cflavorparone^\vee, \cH)-mod. \end{equation}
In particular, via the sheaf-Fukaya comparison of \cite{GPS3}, 
there is an equivalence 
$$\Fuk(\mht_{\beta,0}) \cong  \Gamma(\cflavorparone^\vee, \cB(\cflavorparone^{\vee}, \cH)-mod).  $$
\end{theorem}

The above result is half of a mirror symmetry equivalence. The other half is as follows.

Consider the singular variety $\mht_{1, 0}$. For any sufficiently sufficiently generic $\theta \in \frak{g}^{\vee}_\Z$, we have a crepant resolution $\mht_{1, \theta} \to \mht_{1,0}$.

\begin{theorem} \label{GMW B-model} \cite{gammage-mcbreen-webster}. 
There is an equivalence of dg categories 
\[ \cB(\cflavorparone^{\vee}, \cH)-mod \cong \Coh(\mht_{1, \theta}).\]
This equivalence moreover intertwines the central actions of $\cO(\flavor^{\vee})$
arising from \eqref{central action on B} on the LHS and \eqref{multiplicative moment map} on the RHS.
\end{theorem}

\begin{remark}
The equivalence in \cite{gammage-mcbreen-webster} is most naturally formulated as 
$\Fuk(\mht_{\parone, 0}) \cong \Coh(\mht_{1, \parone})$,
where the RHS is defined as a certain non-commutative resolution of the singular variety $\mht_{1,0}$, depending on the parameter $\parone$. This NCR admits a canonical equivalence with $\Coh(\mht_{1, \theta})$. We will not take this perspective here.
\end{remark}

\section{Lagrangian skeleton of punctured (co)tangent bundle}
\label{section downstairs skeleton}

Let $M$ be a smooth manifold and $T^*M$ its cotangent bundle. Recall that the standard Liouville structure on $T^*M$ has  skeleton given by the zero section, which we will denote $\Lambda_M$.  

Let $H \In M$ be an embedded co-oriented hypersurface. Then there exists a tubular neighborhood $U$ of $H$ and a trivialization $U \cong H \times (-1,1)$. Thus,  we can embed $T^*H \into T^*M$
$$ T^*H \cong T^*H \times \{0,0\} \In T^*H \times T^*(-1,1) \cong T^*U \In T^*M. $$

Let $H_1, \cdots, H_N \In M$ be a collection of embedded co-oriented hypersurfaces that intersects transversely. We can construct a collection tubular neighborhoods $U_i$ of $H_i$, with compatible trivializations in the following sense. 

Let $[N]=\{1,\cdots, N\}$, and $H_I = \cap_{i \in I} H_i, U_I = \cap_{i \in I} U_i$. Let $\cI$ denote the collection of $I$ such that $H_I$ is not empty. 
\begin{definition}
    A compatible system of framed tubular neighborhoods for $\{H_i\}$ is the  data that, for each $i \in [N]$, a tubular neighborhood $U_i$ of $H_i$ and a diffeomorphism ('framing') 
    $$ (\pi_i,h_i): U_i \congto H_i \times (-1,1),  $$
    where $H_i = h_i^{-1}(0)$. 
    The data satisfies that for any $I \In [N]$, we have a diffeomorphism
    $$ (\pi_I, h_I): U_I \congto H_I \times (-1,1)^I$$
    where $h_I$ has components $(h_i)_{i \in I}$, and $\pi_I$ is the composition (in any order) of $\pi_i$ for all $i \in I$. 
\end{definition}

\begin{lemma}
There exists a compatible system of framed neighborhoods for $\{H_i\}$. 
\end{lemma}
\begin{proof}
Since each $H_i$ is a co-oriented and embedded hypersurface, we can construct a tubular neighborhood $U_i$ of $H_i$ with a fibration $h_i: U_i \to (-1,1)$ such that $H_i$ is the fiber over $0$. By shrinking $U_i$, we may assume that if $I \notin \cI$, then  $U_I = \emptyset$. 

Since the collection of $H_i$ intersect transversely, for any  $I=\{i_1,\cdots,i_k\} \in \cI$, we have 
$$ h_I=(h_{i_1},\cdots,h_{i_k}): U_I \to (-1,1)^I $$
where $\vec 0 = (0,\cdots, 0)$ is a regular value of $h_I$. By shrinking $U_i$ and rescale $h_i$, we may assume $h_I$ is a submersion. 

Next, we  build the projection $\pi_I: U_I \to H_I$ inductively. We use $\cJ$ to denote the collection of indices $I$ where $\pi_I$ is not yet specified, hence initially $\cJ = \cI$. Pick a maximal element $I$ in $\cJ$, and consider the submanifold $H_I$ and its neighborhood $U_I$. Let $\cI_{>I}=\{J \in \cI \mid J\supset I\}$. We have $h_I: U_I \to (-1,1)^I$ with $H_I$ as the fiber over $\vec 0$. For any  $J \in \cI_{>I}$, by induction hypothesis, we have $\pi_J$ specified on $U_J$. Hence we can specify $\pi_I|_{U_J}$ as
$$\pi_I|_{U_J} : = (\pi_J, h_{J \RM I}): U_J \to H_J \times (-1,1)^{J \RM I} \cong H_I \cap U_J. $$

It remains to extend $\pi_I$ to the rest of $U_I$. We will use a gradient flow to do this. Define the function $f_I: (-1,1)^I \to \R$ by summing up the coordinates squared.  Let $F_I = h_I^* f_I: U_I \to \R$. Next we construct a Riemannin metric $g_J$ for all $U_J$ with $J \in \cI_{>I}$ inductively, such that $g_J$ is product form for the trivialization $(\pi_J, h_J): U_J \congto H_J \times (-1,1)^J$ and where on $(-1,1)^J$ we use the Euclidean metric. Suppose $J \in \cI_{>I}$ is a maximal element with $g_J$ unspecified, then we can first construct the metric on $H_J$ by extending the metrics on $\cup_K H_J \cap U_K$ for $J \In K \in  \cI$, then the metric on $U_J$ is uniquely specified by the product form requirement. Finally, we construct a metric $g_I$ for $U_I$ by extending the existing metrics $g_J$ on $U_J$ for $J \In I$. Then, we may run a downward gradient flow for function $F_I$ on the metric space $(U_I, g_I)$, which has Morse-Bott critical manifold $H_I \In U_I$.  We define $\pi_I(x)$ to be the limit of $x$ under the flow, this is compatible with the previously specified $\pi_I|_{U_J}$ thanks to the product form of the metric $g_J$ on $U_J$ for $J \supset I$. This finishes the construction of $\pi_I$. 

\end{proof}

Choose and fix a compatible system of framed neighborhoods $U_i$ as in the Lemma, then we have embedding $T^*H_i \into T^*M$. The goal of this section is to endow the punctured cotangent bundle $$
X:= T^*M \RM \cup_{i=1}^N T^*H_i $$ with a Weinstein structure and an explicit Weinstein skeleton. 

First we consider a local model, where $M = \R$ and $H=\{0\}$. 
Let $(x,y)$ be the Cartesian coordinate for $T^*\R$, and let $(r, \theta)$ be the polar coordinate. Let $\rho = -\log r$. We will consider two Liouville 1-forms
$$ \lambda_0 = y dx, \quad \omega_0 = -d\lambda_0 = -dy \wedge dx =  r dr \wedge d \theta $$
$$ \lambda_1 = \rho d \theta, \quad \omega_1 = -d\lambda_1 = -d \rho \wedge d \theta = r^{-1} d r \wedge d \theta.  $$
Let $C_r$ be the radius $r$ circle centered at the origin in $\R^2$. 
\begin{proposition}
    There exists a Liouville structure $\lambda$ on the punctured cotangent bundle 
    $$X := T^*M \RM T^*H = T^*\R \RM \{(0,0)\}, $$
    such that for some $0< \epsilon < 1$ and $c>0$ we have
    $$ \lambda = \begin{cases}
        \lambda_0 & r >  2- \epsilon \\
        c \lambda_1 & r < 1 + \epsilon. 
    \end{cases} $$
    and the skeleton for $(X,\lambda)$ is 
    $$ \Lambda_{-O-}  = (-\infty, -1) \cup C_1 \cup (1, \infty). $$
\end{proposition}

\begin{proof}
We will do a partition of unity for $ (0,\infty) = (0,2) \cup (1,\infty)$. Let $\eta(r)$ be any smooth weakly monotone increasing function which satisfies 
$$ \eta(r) = \begin{cases}
        1 & r \geq  2- \epsilon \\
        \eta(r) \in (0,1), \eta'(r)>0 & r \in (1 + \epsilon, 2- \epsilon) \\
        0 & r \leq 1 + \epsilon. 
    \end{cases} $$
We claim that, for small enough $c>0$, the combination
$$ \lambda = \lambda_0 \eta(r) + c \lambda_1 (1-\eta(r))$$
is a Liouville 1-form. We only need to check that $\omega \neq 0$ anywhere. 
$$ \omega = -d \lambda = (r \eta + \eta'(r) r^2 \sin^2 \theta + c ((1-\eta(r))\log r)') dr \wedge d\theta =: F(r,\theta) dr \wedge d\theta. $$
Over $r \in (0,1+\epsilon]$ and $r \in [2 - \epsilon,\infty)$, $\eta$ is constant, and $\omega \neq 0$. Over $r \in (1+\epsilon, 2-\epsilon)$, we have $(1-\eta(r))\log r)'$ is smooth and bounded, and is positive when $r = 1+\epsilon$
$$ (1-\eta(r))\log r)'|_{r = 1+\epsilon} = 1/(1+\epsilon) > 0$$
hence for small enough $c>0$, we have
$$ r \eta + c (1-\eta(r))\log r)' > 0, \quad r \in [1+\epsilon, 2-\epsilon]. $$
This shows that $F>0$ for all $r$, and hence $\omega$ is non-degenerate. 

Finally, we consider the skeleton of $\lambda$. Note that
$$ \lambda = (-r^2 \sin^2 \theta \eta(r) - c \log(r) (1-\eta)) d\theta + r\sin \theta \cos\theta \eta(r) dr := \lambda_\theta d\theta + \lambda_r dr$$
so the downward Liouville vector field $X_\lambda$, defined by $\iota_{X_\lambda} \omega = \lambda$,  
$$ X_\lambda = F^{-1} \lambda_\theta \d_r - F^{-1} \lambda_r \d_\theta.$$
Part of the skeleton consist of the critical loci of the flow $X_\lambda$, i.e. 
$$\{ (r,\theta) \mid \lambda = 0\} = \{ r \in [2-\epsilon, +\infty), \theta \in \{0, \pi\}\} \cup \{ r =1 \}. $$
The remaing part of the skeleton consist of unstable manifold between the critical loci, we can check that it consists of the intersection of the $x$-axis with the annulus $r \in (1, 2-\epsilon)$. 
\end{proof}

For $H \In M$ a co-oriented hypersurface and $U \cong H \times (-1,1)$ a framed tubular neighborhood, we define the abstract skeleton for $T^*M \RM T^*H$
$$ \Lambda_{M,H}:=  \overline{M \RM H} \cup_{H \times \{-1,1\}} H \times S^1 $$
where $\d (\overline{M \RM H}) = H \times \{-1,1\}$, and $\{-1,1\} \into S^1$ is given by identifying $S^1 \cong \{|z|=1\}$. We have a natural map $p: \Lambda_{M, H} \to M$, where the fiber over any point in $H$ is a $S^1$. 

For a collection of hypersurfaces $H_1, \cdots, H_N$ and framed tubular neighborhood $U_1, \cdots, U_N$, we define the total `abstract' skeleton as
$$ \Lambda_{M, \{H_i\}} := \Lambda_{M, H_1} \times_M \Lambda_{M,H_2} \times \cdots \times \Lambda_{M, H_N}. $$

\begin{proposition} \label{p:bubble-skeleton}
Let $M$ be a smooth manifold, $H_1, \cdots, H_N$ be co-oriented smooth hypersurface intersecting transversely, and let $U_i \cong H_i \times (-1,1)$ be compatibly framed tubular neighborhood of $H_i$. 
Let $X = T^*M \RM \cup_i T^*H_i$. Then there exists a Liouville structure $\lambda$ on $X$ whose skeleton is homeomorphic to the abstract skeleton $\Lambda_{M, \{H_i\}}$, and the skeleton is locally of the product form $\Lambda_{-O-}^k \times \R^{\dim M-k}$.  
\end{proposition}

\begin{proof}

By rescaling, we may assume $(\pi_i, h_i): U_i \cong H_i \times (-3,3)$, and let
$$ U_{i,r} = \{ x \in U_i \mid |h_i(x)| < r\}, \quad V_i = M \RM \overline{U_{i,2}}, $$
then we have an open cover $M = U_i \cup V_i$. We consider the intersection all the open covers for $i=1,\cdots, N$, for any subset $I \In [N]$, we have
$$ W_I = \bigcap_{i=1}^N  \begin{cases}
    U_i & i \in I \\
    V_i & i \notin I
\end{cases}
$$
Then, $W_I$ is an open subset of $U_I$.  Using the framing isomorphism of $U_I \cong H_I \times (-3,3)^I$, we have $W_I \cong H_I^o \times (-3,3)^I$, where $H_I^o = H_I \cap W_I$. We equip 
$$ X_I:=X \cap T^*W_I \cong T^*(H_I^o) \times (T^*(-3,3) \RM \{0\} )^I$$ 
with the product Liouville structure, where we use standard Liouville structure on $T^*(H_I^o)$ and the modified Liouville structure on the puncture space $T^*(-3,3) \RM \{0\}$. One can check that such Liouville structure on the open patches $X_I$ are compatible, and result in the desired skeleton.
\end{proof}

Having determined the skeleton, we have by \cite{GPS3} an equivalence 
\begin{equation}
 \Fuk(T^*M \setminus \bigcup T^* H_i) = \Gamma(\Lambda_{M, \{H_i\}},  \mu sh_{\Lambda_{M, \{H_i\}}} ).
\end{equation}
Strictly speaking, the definition of the Fukaya or microsheaf categories depend on a choice of `Maslov data'.  Here we take the data induced from the fiber polarization on the cotangent bundle. 

It is straightforward to compute the RHS: 

\begin{proposition} \label{downstairs microsheaves} 
Consider the map $\pi : \Lambda_{M, \{H_i\}} \to M$.  Then there is an equivalence of sheaves of categories
$$\pi_* \mu sh_{\Lambda_{M, \{H_i\}}} = \cB_0(M, \bigcup H_i).$$
In particular, passing to global sections gives: 
$$
 \Fuk(T^*M \setminus \bigcup T^* H_i) = \Gamma(\Lambda_{M, \{H_i\}},  \mu sh_{\Lambda_{M, \{H_i\}}} )= \Gamma(M, \cB_0(M, \bigcup H_i)).$$
\end{proposition}
\begin{proof}
    This is locally a product of the one-dimensional calculation of microsheaves on  
    $\Lambda_{T^* \R \setminus (0, 0)}$.  
    This is straightforward and known to have the asserted value; 
    see e.g. \cite{nadler-pants}. 
\end{proof}

\section{Proof of  Theorem \ref{multiplicative main theorem}}
\label{section proof of main theorem}

Fix as in Section \ref{section mht} the defining data of a multiplicative hypertoric variety, and in particular, the integral hyperplane arrangement $\cH \subset \cflavorparone^{\vee}$
from \eqref{hypertoric hyperplane arrangement}.  Note the integer lattice $L$ is identified with the characters of $\flavor^\vee$, and so the group-ring of this lattice is identified with the functions on $\flavor^{\vee}$.  

By \Cref{global algebra cosheaf reduction}, we have 
$$ \cB(\cflavorparone^\vee, \cH)-mod \otimes_{\Z[L]} \Z = \cB_0(\cflavorparone^\vee, \cH)-mod$$

We take global sections, and apply the B-model  calculation from \cite{gammage-mcbreen-webster} (recalled in \Cref{GMW B-model}) on the LHS,  and apply \Cref{downstairs microsheaves} on the RHS.  This
gives: 
$$\Coh(\mht_{1, \theta})\otimes_{\Z[L]} \Z = \Fuk(\flavorparone^\vee \setminus \bigcup H_i).$$

Recall the moment map 
$\mu^{\flavor} : \mht_{1, \theta} \to \flavor^{\vee}$, and the identification of the above $\Z[L]$ with the functions on $\flavor^{\vee}$.  It follows formally that 
$\Coh(\mht_{1, \theta})\otimes_{\Z[L]} \Z  = \Coh(\mu^{-1}(1))$.
This completes the proof of Theorem \ref{multiplicative main theorem}. $\qed$

\section{Group actions and the proof of Theorem \ref{theorem right action}} \label{section right action}

More or less by definition, if $G$ is a topological group, then a `continuous' $G$ action on $C$ is a locally constant sheaf over $BG$ with stalk $C$.  Recall also that we can describe locally constant sheaves on any (sufficiently nice) space $Z$ as follows.  Consider the category $\Delta(Z)$ of singular simplices in $Z$, with morphisms given by inclusions of boundary simplices.  Then local systems on $Z$ are the same as functors out of $\Delta(Z)$ carrying all morphisms to isomorphisms.  

For example, suppose $T$ is a topological space and $G$ a group acting on $T$.  Let us write $Sh_{[G]}(T) \subset Sh(T)$ for the full subcategory on objects which are locally constant along the $G$-orbits.
There is a local system of categories over $BG$ whose fiber is $Sh_{[G]}(T)$.   One constructs it as follows.  We will regard $BG$ as classifying bundles with connection (note the space of connections is contractible), so a simplex $\Delta \in \Delta(BG)$ is the same as a $G$-bundle with connection over the underlying simplex $\Delta$.  

We define a functor $\Delta(BG) \to Categories$ by sending such a simplex $\Delta$ to the category of sheaves over the total space of the corresponding $T$-bundle which are fiberwise in $Sh_{[G]}(T)$ and are locally constant over $T$ with respect to parallel transport along the induced $G$-connection.   Restriction of sheaves provides the structure maps giving a functor from $\Delta(BG)$.  The restriction maps evidently induce equivalences of categories and hence the resulting functor describes a locally constant sheaf of categories over $BG$.  We may also restrict the entire construction to e.g. the full subcategory of $Sh_{[G]}(T)$ constructible with respect to a some particular $G$-invariant stratification of $T$, and in particular, always to the category of local systems on $T$.  

In general, given a locally constant sheaf of categories over $BG$ with stalk $C$, one gets maps $G \to \End(C)$ and $\Omega G  \to \End(1_C)$.  In particular, $\Omega G$ acts on every object in $C$.  In the case
above of $Sh_{[G]}(T)$, this action is computed more directly as follows.  Given a loop in $G$, form the trivial $G$-bundle over $S^1$, with connection whose parallel transport implements the loop in $G$.  The induced connection on the trivial $T$-bundle over $S^1$ allows to parallel transport objects in $Sh_{[G]}(T)$.  Said parallel transport induces an automorphism of the object.  

\begin{example}
    Consider $G = T = S^1$.  Then $Sh_{[S^1]}(S^1)$ is the category of local systems on $S^1$.  The action of $1 \in \Z = \Omega S_1$ on a local system is, per the above description, given by parallel transport along a full rotation of the circle, or in other words, by the monodromy of the local system. 
\end{example}

More generally, given $F \in Sh_{[G]}(T)$ and a loop $\gamma \in \Omega G$, $\gamma: F \to F$ is given on stalks $\gamma: F_x \to F_x$ as the monodromy of the local system $F|_{\gamma \cdot x}$ on $\gamma \cdot x \cong S^1$. 

\begin{example} \label{equivariance calculation}
    Consider $G=\C^*$, $T = \C$.  Then, 
    we recall Beilinson's identification 
    $Sh_{[\C^*]}(\C) = \mathcal{B}-mod$ \cite{Beilinson-glue}, sending the nearby and vanishing cycles at the origin to the nodes at the quiver, and the $t$ and $\tau$ for their respective monodromies.  In particular, $1 \in \Z = \Omega \C^*$ acts by multiplication on the nodes by $t$ and $\tau$, respectively.  
\end{example}

Finally, we recall that the category of $G$-equivariant sheaves can be computed as
$$Sh(T)^G \cong Sh_{[G]}(T) \otimes_{C_* \Omega G} \Z$$

\vspace{2mm}

We want an analogous construction where $G = Aut(X)$ is some enhanced version of the exact symplectomorphism group of a Liouville manifold $X$; the enhancement concerns the topological information needed to define Fukaya or microsheaf categories (i.e. what is called Maslov data in \cite{nadler-shende}).  
In this case, the points of $B Aut(X)$ are then exact symplectomorphisms of $X$; the 1-simplices are paths of symplectomorphisms, etc.  

One such construction was introduced by Oh and Tanaka in \cite{oh-tanaka-smooth, oh-tanaka}, by generalizing results around Floer homology (and wrapping) to the case of bundles of Liouville manifolds over a base.  One important technical innovation was to show that, for the purpose of describing locally constant sheaves of categories over $B Aut(X)$, one can work with a restricted collection of simplices 
$\Delta^{sm}(B Aut(X))$, consisting of those for which the symplectomorphsims are smoothly varying and have suitably trivialized collars near the boundary.  Here we will use a different, though presumably equivalent, approach; ours is chosen to interact more transparently with the sheaf-Fukaya equivalence of \cite{GPS3}.  

The task is to construct a functor from $\Fuk: \Delta^{sm}(B Aut(X)) \to A_\infty-cat$, which in particular sends 0-simplices to the Fukaya category of $X$, and all morphisms in $\Delta^{sm}(B Aut(X))$ to isomorphisms.  We proceed as follows.  An element $\Delta \in \Delta^{sm}(B Aut(X))$ has associated $X$-bundle $X_\Delta \to \Delta$ with connection given by Hamiltonian symplectormorphisms.  The Hamiltonian can be used to produce a canonical exact symplectic structure on the pullback of the $X$-bundle to $X_{T^* \Delta} \to T^* \Delta$ such that symplectic parallel transport along $\Delta$ realizes the original connection.  Note that $X_{T^* \Delta}$ is a Liouille sector with corners. We define $\Fuk(\Delta) := \Fuk(X_{T^* \Delta})$.  
For $\Delta' \to \Delta$ the inclusion of a collar neighborhood of a boundary stratum, the boundary trivializations involved in the definition of $\Delta^{sm}$ ensure an inclusion of sectors $X_{T^* \Delta'} \to X_{T^* \Delta}$; it is evident that the \cite{GPS1} covariant functor $\Fuk(X_{T^* \Delta'}) \to \Fuk(X_{T^* \Delta})$ is an equivalence.  This defines the desired functor, along with checking that it determines a local system of categories over $B Aut(X)$.  We denote this local system by $\underline{\Fuk}(X)$.    
Let us remark that the same construction applies to a stopped Liouville manifold $(X, \mathfrak{f})$ and the exact symplectomorphisms which preserve the stop, $Aut(X, \mathfrak{f})$.  

For $G \subset Aut(X)$, we may either perform the same construction restricted to $G$, or equivalently pull back along $BG \to BAut(X)$, to obtain a local system of categories over $BG$, which we will also denote $\underline{\Fuk}(X)$.  Let us observe that the formation of $\underline{\Fuk}(X)$ is evidently compatible with sectorial descent \cite{GPS2} for a $G$-invariant sectorial cover. 

Suppose $X$ is sufficiently Weinstein.  Then from \cite{nadler-shende}, we may associate a category of microsheaves on the skeleton of $X$, which we denote $\mathfrak{Sh}(X)$; note e.g. $\mathfrak{Sh}(T^*M) = \Loc(M)$.  More generally, if 
$(X, \mathfrak{f})$ is sufficiently Weinstein, there is a 
corresponding category of microsheaves on the relative skeleton. 
These categories were shown to match the corresponding Fukaya categories in \cite{GPS3}.  
 
If we restrict the above discussion of equivariant structures to constructions preserving the property of sufficiently Weinstein, we may apply $\mathfrak{Sh}$ to obtain a sheaf of categories $\underline{\mathfrak{Sh}}(X)$ over $B Aut^{Wein}(X)$.  By functoriality of \cite{GPS3}, 
$$\underline{\mathfrak{Sh}}(X) \cong \underline{\Fuk}(X)|_{B Aut^{Wein}}$$  While $Aut^{Wein}(X)$ is poorly understood, it certainly contains the exact symplectomorphisms whose action preserves the Liouville vector field, which will suffice for our purposes. 

Finally, let us consider the case $X = T^* M$.  Then we have the natural $Diff(M)$ action on $T^*M$; note it preserves the Liouville flow, and thus we may consider $\underline{\mathfrak{Sh}}(T^* M)|_{BDiff(M)}$.  On the other hand, the constructions on sheaf categories described in the opening paragraphs of this section lead to a local system of categories $\underline{\mathrm{Loc}}(M)$ over $BHomeo(M)$, with fiber $\mathrm{Loc}(M)$.  It follows directly from the constructions that
$$\underline{\mathfrak{Sh}}(T^* M)|_{BDiff(M)} \cong \underline{\mathrm{Loc}}(M)|_{BDiff(M)}$$
More generally, if we fix a Whitney stratification $\mathcal{S}$ and write $Diff(M, \mathcal{S})$ for diffeomorphisms preserving the stratification and $N^* \mathcal{S}$ for the union of conormals to the strata, then we have similarly
$$\underline{\mathfrak{Sh}}(T^* M, N^* \mathcal{S} )|_{BDiff(M, \mathcal{S})} \cong \underline{\mathrm{Sh}}_{\mathcal{S}}(M)|_{BDiff(M, \mathcal{S})}$$

These constructions and comparisons in place, we proceed to the proof of Theorem \ref{theorem right action}:

\begin{proof}[Proof of Theorem \ref{theorem right action}]
    As described in Section \ref{section mht}, there is a Hamiltonian action of $U(1)^n$ on $\mht_{\parone,0}$.  From the above considerations, we obtain an enrichment 
    of $\Fuk(\mht_{\parone,0})$ by 
    $C_* \Omega U(1)^n = \Z[s_1^{\pm}, \ldots, s_n^{\pm}]$.  On the other hand, we have a similar enrichment of $\Coh(\mht_{1, \theta})$ via pulling back functions under the moment map.  The task is to show that these enrichments match under the mirror symmetry $\Fuk(\mht_{\parone,0}) \cong \Coh(\mht_{1, \theta})$ of \cite{gammage-mcbreen-webster} (recalled as Theorems \ref{GMW mush to quiver} and \ref{GMW B-model} above).  By the compatibility asserted in Theorem \ref{GMW B-model}, it suffices to show that, under the equivalence of Theorem \ref{GMW mush to quiver}, the geometrically constructed action on the Fukaya category is carried to the explicitly described central action from the quiver description \eqref{central action on B}.  
    
    We proceed as follows.  The space $\mht_{\parone,0}$ admits a sectorial cover compatible with the $U(1)^n$ action, such that the skeleton of each stratum is identified as some
    $(\C^* \subset T^* \C^*)^a \times (\C \sqcup T^*_0 \C \subset T^* \C)^{n-a}$, with the $U(1)^n$ action lifted from the base.   From the above discussion, we may compute the action of $C_* \Omega U(1)^n$ locally on such charts, and then carry out the calculation after passing to sheaves.  Now the result follows from taking products of Example \ref{equivariance calculation}.
\end{proof}

\newpage

\bibliographystyle{plain}
\bibliography{refs}

\end{document}